\documentclass[12pt,oneside,english]{amsart}
\usepackage[T1]{fontenc}
\usepackage[latin9]{inputenc}
\usepackage{amstext}
\usepackage{amsthm}

\makeatletter
\theoremstyle{plain}
\newtheorem{thm}{\protect\theoremname}
\theoremstyle{definition}
\newtheorem*{problem*}{\protect\problemname}
\theoremstyle{plain}
\newtheorem{fact}[thm]{\protect\factname}
\theoremstyle{plain}
\newtheorem{lem}[thm]{\protect\lemmaname}
\theoremstyle{plain}
\newtheorem{prop}[thm]{\protect\propositionname}

\makeatother

\usepackage{babel}
\providecommand{\factname}{Fact}
\providecommand{\lemmaname}{Lemma}
\providecommand{\problemname}{Problem}
\providecommand{\propositionname}{Proposition}
\providecommand{\theoremname}{Theorem}

\begin{document}
\title{Forbidden conductors and sequences of $\pm1$s}
\author{Maciej Radziejewski}
\begin{abstract}
We study ``forbidden\textquoteright \textquoteright{} conductors,
i.e. numbers $q>0$ satisfying algebraic criteria introduced by J.~Kaczorowski,
A.~Perelli and M.~Radziejewski, that cannot be conductors of $L$-functions
of degree $2$ from the extended Selberg class. We show that the set
of forbidden $q$ is dense in the interval $(0,4)$, solving a problem
posed in~\cite{ForbiddenConductors}. We also find positive points
of accumulation of rational forbidden $q$.
\end{abstract}


\maketitle
The Selberg class $\mathcal{S}$ and the extended Selberg class $\mathcal{S}^{\#}$
are axiomatically defined classes of $L$-functions, cf., e.g.,~\cite{NewSurvey}.
Each $F\in\mathcal{S}^{\#}$ has a degree $d_{F}\geq0$ and a conductor
$q_{F}>0$. It was shown by J.~Kaczorowski and A.~Perelli~\cite{Structure1,Structure2},
cf. also~\cite{Degree0}, that if $d_{F}<2$, then $d_{F}$ and $q_{F}$
are integers. Moreover, if $d\in\left\{ 0,1\right\} $ and $q$ is
a positive integer, then there exists an $L$-function $F\in$$\mathcal{S}^{\#}$
with $d_{F}=d$ and $q_{F}=q$~\cite[Theorems 1 and 2]{Structure1}.
As pointed out by V.~Blomer, Theorem~3 of~\cite{Structure1} implies
that there is no $F\in\mathcal{S}$ with $d_{F}=1$ and $q_{F}=2$.
The structure of 
\[
\mathcal{S}_{2}^{\#}=\left\{ F\in\mathcal{S}^{\#}:d_{F}=2\right\} 
\]
is a subject of current research~\cite{Degree2Conductor1}. J.~Kaczorowski,
A.~Perelli and M.~Radziejewski~\cite{ForbiddenConductors} have
shown the existence of $F\in{\mathcal{S}}_{2}^{\#}$ with $q_{F}=q$
for every real $q\geq4$ and for $q=4\cos^{2}(\pi/m)$ with integer
$m\geq3$~\cite[Lemma 6]{ForbiddenConductors}. They also considered
fractions
\begin{equation}
c(q,{\bf m})=m_{k}+\cfrac{1}{qm_{k-1}+\cfrac{q}{qm_{k-2}+\cfrac{q}{\ddots\,+\cfrac{q}{qm_{0}}}}}\label{eq:cqm}
\end{equation}
where $q>0$ and ${\bf m}=(m_{0},\dots,m_{k})\in\mathbf{Z}^{k+1}$.
For a given $q$ the sequence ${\bf m}$ is called a path if~(\ref{eq:cqm})
is well defined, i.e. there are no zeros in denominators. It is called
a loop if $c(q,{\bf m})=0$. A sequence ${\bf m}$ is therefore a
path if and only if none of $(m_{0})$, $(m_{0},m_{1})$, \ldots ,
$(m_{0},\dots,m_{k-1})$ is a loop. The weight of ${\bf m}$ is then
defined as
\begin{equation}
w(q,{\bf m})=q^{k/2}\prod_{j=0}^{k-1}\left|c(q,{\bf m}_{j})\right|,\label{eq:wqm}
\end{equation}
where ${\bf m}_{j}=(m_{0},\dots,m_{j})$ for $0\leq j\leq k$. We
also write $c(q,m_{0},\dotsc,m_{n})$ and $w(q,m_{0},\dotsc,m_{n})$
where appropriate. The set of loops for a given $q$ is denoted as
$L(q)$. In particular, $\mathbf{m}=(0)\in L(q)$ and $w(q,\mathbf{m})=1$.
It was shown in~\cite[Theorem 1]{ForbiddenConductors} that if $q=q_{F}$
is a conductor of an $L$-function $F\in{\mathcal{S}}_{2}^{\#}$,
then 
\begin{equation}
w(q,\mathbf{m})=1\text{ for all }\mathbf{m}\in L(q),\label{eq:unit-weight-on-L(q)}
\end{equation}
and that (\ref{eq:unit-weight-on-L(q)}) implies that $w(q,\mathbf{m})$
is a function of $c(q,\mathbf{m})$, i.e. $w(q,\mathbf{m})=w(q,\mathbf{m}')$
for all paths $\mathbf{m},\mathbf{m}'$ satisfying $c(q,\mathbf{m})=c(q,\mathbf{m}')$.
In the present paper, for brevity, we call $q$ forbidden if~(\ref{eq:unit-weight-on-L(q)})
is false. 

In principle, to show that a fixed $q\in(0,4)$ is forbidden, it is
enough to find a loop $\mathbf{m}\in L(q)$ with $w(q,\mathbf{m})\neq1$.
For rational $q=a/b$ this process bears some similarity to Euclid's
Algorithm, but in general we have no deterministic algorithm to find
such a loop. The search for loops of non-unit weight becomes harder
for large $a$ and, in particular, when $q$ gets closer to $4$.
It was shown in~\cite{ForbiddenConductors} that the following $q$s
are forbidden:
\begin{equation}
q=\frac{4}{n}\cos^{2}(\pi\ell/(2k+1))\label{eq:forbidden-set-old}
\end{equation}
with integers $k,\ell,n$ such that $k\geq1$, $1\leq\ell<2k+1$,
$(\ell,2k+1)=1$, $n\geq2$. The $q$s satisfying~(\ref{eq:forbidden-set-old})
are dense in the interval $0<q<2$. In addition, some rational $q$s
were also determined to be forbidden:
\begin{itemize}
\item infinitely many with $0<q<1$ (with a point of accumulation in $0$),
\item 16271 $q$s with $1<q<2$,
\item 3865 $q$s with $2<q<3$,
\item 293 $q$s with $3<q<4$,
\end{itemize}
cf.~the online table of computation results accompanying~\cite{ForbiddenConductors}.
In the present paper we solve one of the problems posed in~\cite{ForbiddenConductors}
and construct a set of forbidden $q$ that is dense in $(0,4)$. We
also show that the set of forbidden rational $q$s has positive points
of accumulation, including $\frac{3\pm\sqrt{5}}{2}$.
\begin{thm}
\label{thm:forbidden-qs-are-dense}The set of forbidden $q$s is dense
in the interval $(0,4)$.
\end{thm}

\begin{thm}
\label{thm:forbidden-rational-qs-accumulate}The set of accumulation
points of forbidden rational $q$s contains $\frac{3-\sqrt{5}}{2}\cong0.381966$
and $\frac{3+\sqrt{5}}{2}\cong2.618034$.
\end{thm}

We start by listing further definitions and notation to be used in
the paper. Next we show that, when $q$ is close to $4$, a loop $\mathbf{m}\in L(q)$
must contain a long chain of alternating $\pm1$s, except for some
trivial loops, that always have a unit weight. This fact will not
be used directly in our main result, but it serves to set the scene,
explaining why we consider loops based on such chains.

In Section~\ref{sec:Proof1} we prove our main result, Theorem~\ref{thm:forbidden-qs-are-dense}.
The proof is based on the study of loops of the form
\begin{equation}
(m_{0},\dotsc,m_{n})=(1,-1,\dotsc,(-1)^{n-1},(-1)^{n}+c).\label{eq:pm1+c}
\end{equation}
If $c=0$ or $c=(-1)^{n+1}$, such loops must have unit weight, however,
for other values of $c$ this is not so, cf. Lemma~\ref{lem:non-unit-pm1+c}.
Using the fact that $c(q,\mathbf{m})$ is a rational function of $q$
and the Darboux property we construct a dense set of forbidden $q$s.

In Section~\ref{sec:Proof2} we show Theorem~\ref{thm:forbidden-rational-qs-accumulate}
by finding rational $q$s for which~(\ref{eq:pm1+c}) with $n=4$
is a loop. The problem comes down to finding integer values of a given
rational function at rational arguments $a/b$. The denominator of
our rational function can be interpreted as the norm of an element
$a-b\omega$ in an algebraic number field $\mathbf{Q}(\omega)$. Hence
we only need to find enough units of the form $a-b\omega$, as for
these the value of the function will be an integer. This approach
strictly depends on loop length being equal to $4$, as in that case
the field $\mathbf{Q}(\omega)$ is quadratic, so the general form
of units coincides with the form of factors in the decomposition of
a homogeneous polynomial in $a$ and $b$.

The problem of finding integer values of a rational function seems
to be of independent interest.
\begin{problem*}
Given a rational function $f\in\mathbf{Q}(x)$ determine the set of
$x\in\mathbf{Q}$ such that $f(x)\in\mathbf{Z}$.
\end{problem*}

\section{Preliminaries}

We use the notation $e(x)=e^{2\pi ix}$ and write $\left\lfloor x\right\rfloor $
for the largest integer $\leq x$. Given $q>0$ we call a loop ${\bf m}=(m_{0},\dots,m_{n})$
proper if $n=0$ or $m_{j}\neq0$ for all $j$. It was shown in~\cite{ForbiddenConductors}
that the study of loops, and of condition~(\ref{eq:unit-weight-on-L(q)})
in particular, can be reduced to proper loops, through an equivalence
reminiscent of homotopy. In fact proper loops have a group structure
and weight is a multiplicative homomorphism on proper loops.

Now we define polynomials in variables $\lambda,m_{0},m_{1},\dotsc$
They are related to the numerator and denominator of~(\ref{eq:cqm})
when $q=\lambda^{2}$. Let $\mathcal{I}$ denote the set of finite
subsets of $I\subset\mathbf{N}_{0}$ such that
\[
i\equiv\left|[0,i)\cap I\right|\pmod{2}\quad\text{for all}\quad i\in I.
\]
Let
\begin{multline}
f_{n}=f_{n}(\lambda,m_{0},\dotsc,m_{n-1})=\sum_{\substack{0\leq k\leq n\\
k\equiv n\,(\text{mod }2)
}
}\left(\sum_{\substack{I\in\mathcal{I}\\
I\subseteq[0,n)\\
\left|I\right|=k
}
}\prod_{i\in I}m_{i}\right)\lambda^{k},\\
n\geq0.\label{eq:fn-explicit}
\end{multline}
The polynomials $P_{n}$ and $Q_{n}$ in~\cite[Proof of Theorem 2]{ForbiddenConductors}
are related to $f_{n}$ by $P_{n}(\lambda^{2},\mathbf{m})=\lambda^{n}f_{n+1}(\lambda,\mathbf{m})$
and $Q_{n}(\lambda^{2},\mathbf{m})=\lambda^{n+1}f_{n}(\lambda,\mathbf{m}_{n-1})$.
\begin{fact}
\label{fact:f-degree}$f_{n}$ has degree $n$ with respect to $\lambda$
and the leading coefficient is $\prod_{i=0}^{n-1}m_{i}$. 
\end{fact}

\begin{lem}
\label{lem:fn-recurrence}We have
\[
f_{0}=1,
\]
\[
f_{1}=m_{0}\lambda,
\]
and
\begin{equation}
f_{n+1}=m_{n}\lambda f_{n}+f_{n-1},\qquad n\geq1.\label{eq:fn-recurrence}
\end{equation}
\end{lem}

\begin{proof}
For $n\in I\subset[0,n+1)$, $I\in\mathcal{I}$ and $\left|I\right|\equiv n+1\pmod{2}$,
we have $I=\left\{ n\right\} \cup I'$ with $I'\subset[0,n)$ and
$\left|I'\right|\equiv n\pmod{2}$. If $n\notin I\subset[0,n+1)$,
$I\in\mathcal{I}$ and $\left|I\right|\equiv n+1\pmod{2}$, then also
$n-1\notin I$, as otherwise we would get
\[
n\equiv\left|I\setminus\left\{ n-1\right\} \right|=\left|I\cap[0,n-1)\right|\equiv n-1\pmod{2},
\]
a contradiction. This implies~(\ref{eq:fn-recurrence}). 
\end{proof}
\begin{lem}
If $\lambda>0$ and $\mathbf{m}=(m_{0},\dotsc,m_{n})\in\mathbf{Z}^{n+1}$
is a path for $q=\lambda^{2}$, then
\begin{equation}
c(\lambda^{2},\mathbf{m})=\frac{f_{n+1}(\lambda,\mathbf{m})}{\lambda f_{n}(\lambda,\mathbf{m}_{n-1})}\label{eq:c-by-continuants}
\end{equation}
and
\begin{equation}
w(\lambda^{2},\mathbf{m})=\left|f_{n}(\lambda,\mathbf{m}_{n-1})\right|.\label{eq:w-by-continuants}
\end{equation}
\end{lem}

\begin{proof}
By~(\ref{eq:fn-recurrence}) we have
\begin{align}
\frac{f_{n+1}}{\lambda f_{n}} & =\frac{m_{n}\lambda f_{n}+f_{n-1}}{\lambda f_{n}}\label{eq:fn-ratio-shift}\\
 & =m_{n}+\frac{1}{\lambda^{2}\frac{f_{n}}{f_{n-1}}},\qquad n\geq1\nonumber 
\end{align}
and~(\ref{eq:c-by-continuants}) follows. Then, by~(\ref{eq:wqm})
and~(\ref{eq:c-by-continuants}) we obtain~(\ref{eq:w-by-continuants}).
\end{proof}

\section{A necessary condition for loops}

Suppose $2<q<4$. Let $x_{0}=+\infty$, $x_{1}=1$ and 
\[
x_{n+1}=\begin{cases}
1-\frac{1}{qx_{n}}, & x_{n}\geq\frac{1}{q},\\
0 & \text{otherwise,}
\end{cases}
\]
for $n\geq1$. The sequence $(x_{n})$ is weakly decreasing and bounded,
so it has a limit point. The map $x\mapsto1-\frac{1}{qx}$ has no
fixed point in $(0,+\infty)$, therefore $\lim_{n\to\infty}x_{n}=0$.
Let $C(q)$ denote the largest integer such that $x_{C(q)}\geq\frac{1}{q}$.
\begin{prop}
\label{prop:long-chain}Let $2<q<4$ and let $\mathbf{m}=(m_{0},\dotsc,m_{k})\in L(q)$
be a non-zero proper loop. Let $\ell$ be the smallest non-negative
integer such that $\left|c(q,\mathbf{m}_{j})\right|\leq1$ for $\ell\leq j\leq k$.
Then we have $k\geq\ell+C(q)$ and
\[
m_{j}=(-1)^{j}\varepsilon,\qquad\ell\leq j\leq\ell+C(q)-1,
\]
for some $\varepsilon=\pm1$.
\end{prop}

\begin{proof}
Suppose for some $i\in\left\{ 0,\dotsc,k-1\right\} $ we have $\left|c(q,\mathbf{m}_{i})\right|>\frac{1}{q}$
and $\left|c(q,\mathbf{m}_{i+1})\right|\leq1$. Then we have
\begin{equation}
x_{j}<\left|c(q,\mathbf{m}_{i})\right|\leq x_{j-1}\label{eq:cq-interval}
\end{equation}
for some $1\leq j\leq h$. We have $\left|c(q,\mathbf{m}_{i+1})\right|\leq1$
and 
\[
0<\left|\frac{1}{qc(q,\mathbf{m}_{i})}\right|<\frac{1}{qx_{j}}\leq1,
\]
and the difference
\[
m_{i+1}=c(q,\mathbf{m}_{i+1})-\frac{1}{qc(q,\mathbf{m}_{i})}
\]
is a non-zero integer. This implies that $m_{i+1}=\pm1$, moreover
$\operatorname{sgn}(m_{i+1})=\operatorname{sgn}(c(q,\mathbf{m}_{i+1}))=-\operatorname{sgn}(c(q,\mathbf{m}_{i}))$
and $0<\left|c(q,\mathbf{m}_{i+1})\right|<1$, in particular $i\leq k-2$.
We also have

\[
c(q,\mathbf{m}_{i+1})=\operatorname{sgn}(c(q,\mathbf{m}_{i}))\left(-1+\frac{1}{q\left|c(q,\mathbf{m}_{i})\right|}\right),
\]
therefore
\[
x_{j+1}<\left|c(q,\mathbf{m}_{i+1})\right|\leq x_{j}.
\]

If $\ell>0$, then (\ref{eq:cq-interval}) holds for $i=\ell-1$ and
$j=1$, hence also for $1\leq j\leq C(q)$ and $i=\ell+j-2$. Moreover,
we have
\[
m_{\ell+j-1}=(-1)^{j}\operatorname{sgn}(c(q,\mathbf{m}_{\ell-1})),\qquad j=1,\dotsc,C(q),
\]
and $k\geq\ell+C(q)$.

If $\ell=0$, then $m_{0}=c(q,\mathbf{m}_{0})=\pm1$, in particular
$m_{0}=\operatorname{sgn}(m_{0})$. Now (\ref{eq:cq-interval}) holds
for $j=2$ and $i=0$, hence also for $2\leq j\leq C(q)$ and $i=j-2$.
Moreover, we have
\[
m_{j-1}=(-1)^{j-1}m_{0},\qquad j=2,\dotsc,C(q),
\]
and $k\geq C(q)=\ell+C(q)$.
\end{proof}

\section{\label{sec:Proof1}Proof of Theorem~\ref{thm:forbidden-qs-are-dense}}

Let
\[
g_{n}(\lambda)=f_{n}(\lambda,1,-1,\dotsc,(-1)^{n-1}),\qquad n\geq0.
\]
We are going to consider loops of the form~(\ref{eq:pm1+c}). It
follows from~(\ref{eq:c-by-continuants}) and~(\ref{eq:fn-ratio-shift})
that such loops correspond to integer values of $\frac{g_{n+1}(\lambda)}{\lambda g_{n}(\lambda)}$.
\begin{lem}
\label{lem:gn-roots}We have
\[
g_{n}=\pm\prod_{j=1}^{n}\left(\lambda-\left(e\left(\frac{j}{2n+2}\right)+e\left(-\frac{j}{2n+2}\right)\right)\right).
\]
\end{lem}

\begin{proof}
A part of the argument in this proof was contained in the proof of
Theorem 4 in~\cite{ForbiddenConductors}, but we repeat it because
of the different generality and setup. We have, by~(\ref{eq:fn-explicit}),
\[
g_{n}=\sum_{0\leq\ell\leq n/2}\left(\sum_{\substack{I\in\mathcal{I}\\
I\subseteq[0,n)\\
\left|I\right|=n-2\ell
}
}(-1)^{\left\lfloor (n-2\ell)/2\right\rfloor }\right)\lambda^{n-2\ell}.
\]
The number of sets $I$ in the inner sum is equal to the number of
$(n-2\ell)$-element subsets of $\left\{ 0,2,\dotsc,2n-2\right\} $,
by the bijective mapping where $I=\left\{ a_{0},\dotsc,a_{n-2\ell-1}\right\} $
with $a_{0}<\dotsc<a_{n-2\ell-1}$ corresponds to
\[
I'=\left\{ a_{0},a_{1}+1,\dotsc,a_{n-2\ell-1}+n-2\ell-1\right\} \subseteq\left\{ 0,2,\dotsc,2n-2\ell-2\right\} .
\]
Hence
\[
g_{n}=\pm\sum_{0\leq\ell\leq n/2}(-1)^{\ell}\binom{n-\ell}{n-2\ell}\lambda^{n-2\ell}.
\]
By Fact~\ref{fact:f-degree} it is enough to show that 
\[
g_{n}\left(e\left(\frac{j}{2n+2}\right)+e\left(-\frac{j}{2n+2}\right)\right)=0,\qquad j=1,\dotsc,n.
\]
Fix $j$ and $\omega=e\left(\frac{j}{2n+2}\right)+e\left(-\frac{j}{2n+2}\right)$.
We have
\begin{align*}
g_{n}(\omega) & =\pm\sum_{0\leq\ell\leq n/2}(-1)^{\ell}\binom{n-\ell}{n-2\ell}\sum_{0\leq k\leq n-2\ell}\binom{n-2\ell}{k}e\left(\frac{(2k-n+2\ell)j}{2n+2}\right)\\
 & =\pm\sum_{0\leq m\leq n}e\left(\frac{(2m-n)j}{2n+2}\right)\sum_{0\leq\ell\leq n/2}(-1)^{\ell}\binom{n-\ell}{n-2\ell}\binom{n-2\ell}{m-\ell}\\
 & =\pm\sum_{0\leq m\leq n}e\left(\frac{(2m-n)j}{2n+2}\right)=0,
\end{align*}
where we use the fact that $\binom{n-2\ell}{m-\ell}=0$ when $m-\ell>n-2\ell$,
and the identity
\[
\sum_{0\leq\ell\leq n/2}(-1)^{\ell}\binom{n-\ell}{n-2\ell}\binom{n-2\ell}{m-\ell}=1,\qquad0\leq m\leq n,
\]
which follows by induction.
\end{proof}
\begin{lem}
We have
\begin{equation}
g_{n}^{2}+g_{n+1}g_{n-1}=1,\qquad n\geq1.\label{eq:gn-recurrence}
\end{equation}
\end{lem}

\begin{proof}
By Lemma~\ref{lem:fn-recurrence} we have
\[
g_{0}=1,
\]
\[
g_{1}=\lambda,
\]
and
\[
g_{2}=-\lambda^{2}+1,
\]
so~(\ref{eq:gn-recurrence}) holds for $n=1$. If~(\ref{eq:gn-recurrence})
holds for a certain $n\geq1$, then, using~(\ref{eq:fn-recurrence})
twice, we obtain

\begin{align*}
g_{n+1}^{2}+g_{n+2}g_{n} & =g_{n+1}^{2}+((-1)^{n+1}\lambda g_{n+1}+g_{n})g_{n}\\
 & =g_{n}^{2}+(g_{n+1}-(-1)^{n}\lambda g_{n})g_{n+1}\\
 & =g_{n}^{2}+g_{n+1}g_{n-1}=1,
\end{align*}
so the assertion follows by induction.
\end{proof}
\begin{lem}
\label{lem:non-unit-pm1+c}If $n\geq1$, $c\in\mathbf{Z}\setminus\left\{ 0,(-1)^{n+1}\right\} $
and $q>0$ are such that $\mathbf{m}=(1,-1,\dotsc,(-1)^{n-1},(-1)^{n}+c)\in\mathbf{Z}^{n+1}$
satisfies
\[
\mathbf{m}\in L(q),
\]
then $w(q,\mathbf{m})\neq1$.
\end{lem}

\begin{proof}
It follows from~(\ref{eq:fn-ratio-shift}) and~(\ref{eq:c-by-continuants})
that
\begin{equation}
\frac{g_{n+1}(\sqrt{q})}{\sqrt{q}g_{n}(\sqrt{q})}=\frac{f_{n+1}(\sqrt{q},\mathbf{m})}{\sqrt{q}f_{n}(\sqrt{q},\mathbf{m}_{n-1})}-c=-c.\label{eq:value-as-ratio}
\end{equation}
By~(\ref{eq:gn-recurrence}) and~(\ref{eq:fn-recurrence}) we have
\[
g_{n+1}^{2}(\sqrt{q})+(-1)^{n+1}\sqrt{q}g_{n+1}(\sqrt{q})g_{n}(\sqrt{q})+g_{n}^{2}(\sqrt{q})=1,
\]
so~(\ref{eq:value-as-ratio}) implies
\[
g_{n}^{2}(\sqrt{q})\left(c^{2}q+(-1)^{n}cq+1\right)=1.
\]
Hence, by~(\ref{eq:w-by-continuants}), we have
\begin{align*}
w(q,\mathbf{m}) & =\left|f_{n}(\sqrt{q},\mathbf{m}_{n-1})\right|\\
 & =\left|g_{n}(\sqrt{q})\right|\\
 & =\frac{1}{\sqrt{\left|1+cq(c+(-1)^{n})\right|}}.
\end{align*}
\end{proof}
To complete the proof of Theorem~\ref{thm:forbidden-qs-are-dense}
let
\begin{multline*}
U_{n}=\left\{ \left(e\left(\frac{j}{2n+2}\right)+e\left(-\frac{j}{2n+2}\right)\right)^{2}:1\leq j\leq n,(j,n+1)=1\right\} ,\\
n\geq1.
\end{multline*}
We show that every element of the set
\[
U=\bigcup_{n=1}^{\infty}U_{n}
\]
is an accumulation point of the set of forbidden $q$s. Since $U$
is dense in the interval $[0,4)$ and the set of accumulation points
on the real line is closed, the assertion follows.

Note that the sets $U_{n}$ are pairwise disjoint. Fix $n$ and $t_{0}\in U_{n}$.
Let $t_{1}>t_{0}$ be such that $t_{1}<4$ and
\begin{equation}
(t_{0},t_{1})\cap U_{k}=\emptyset\quad\text{for}\quad k\in\left\{ 1,\dotsc,n-1,n+1\right\} .\label{eq:disjoint-interval}
\end{equation}
Let $\epsilon=\operatorname{sgn}\left(g_{n+1}(t_{1})/g_{n}(t_{1})\right)$.
By Lemma~\ref{lem:gn-roots} and~(\ref{eq:disjoint-interval}) the
rational function 
\[
\frac{g_{n+1}(x)}{\sqrt{x}g_{n}(x)}
\]
has no roots and no singularities in the interval $x\in(t_{0},t_{1})$.
Moreover, by~(\ref{eq:c-by-continuants}) and~(\ref{eq:disjoint-interval}),
every sequence of the form~(\ref{eq:pm1+c}) is a path for all $q\in(t_{0},t_{1})$.
We have
\[
\lim_{x\to t_{0}^{+}}\frac{\epsilon g_{n+1}(x)}{\sqrt{x}g_{n}(x)}=+\infty,
\]
so, by the Darboux property, there exists a sequence $q_{k}\underset{k\to\infty}{\longrightarrow}t_{0}^{+}$
such that
\[
\frac{\epsilon g_{n+1}(q_{k})}{\sqrt{x}g_{n}(q_{k})}=c_{k}\in\mathbf{Z}\cap[3,+\infty),
\]
and in fact $c_{k}\underset{k\to\infty}{\longrightarrow}+\infty$.
It follows from~(\ref{eq:fn-ratio-shift}) and~(\ref{eq:c-by-continuants})
that for $\mathbf{m}\in\mathbf{Z}^{n+1}$, 
\[
\mathbf{m}=(1,-1,\dotsc,(-1)^{n-1},(-1)^{n}-\epsilon c_{k}),
\]
we have $\mathbf{m}\in L(q_{k})$. Hence $q_{k}$ is forbidden by
Lemma~\ref{lem:non-unit-pm1+c}.\hfill{}$\qed$

\section{\label{sec:Proof2}Proof of Theorem~\ref{thm:forbidden-rational-qs-accumulate}}

For $\mathbf{m}=(1,-1,1,-1,c)$ and rational positive $q$ we have,
by~(\ref{eq:fn-explicit}),~(\ref{eq:c-by-continuants}) and~(\ref{eq:w-by-continuants}),
\[
c(q,\mathbf{m})=c+\frac{2-q}{1-3q+q^{2}}
\]
and
\[
w(q,\mathbf{m})=\left|1-3q+q^{2}\right|.
\]
We put $q=a/b$, $(a,b)=1$, and we obtain
\[
c(q,\mathbf{m})=c+\frac{2b^{2}-ab}{b^{2}-3ab+a^{2}}.
\]
Existence of $c\in\mathbf{Z}$ such that $(1,-1,1,-1,c)\in L(q)$
is therefore equivalent to $b^{2}-3ab+a^{2}\mid2b^{2}-ab$ which,
in turn, reduces to
\begin{equation}
b^{2}-3ab+a^{2}=\pm1\label{eq:unit-norm}
\end{equation}
when $(a,b)=1$. If~(\ref{eq:unit-norm}) holds for any integers
$a,b$ (not necessarily relatively prime) such that
\begin{equation}
b^{2}>1\quad\text{and}\quad a/b\neq1,2,\label{eq:sufficient-conditions}
\end{equation}
then 
\[
\mathbf{m}=(1,-1,1,-1,-(2b^{2}-ab)/(b^{2}-3ab+a^{2})),
\]
is a path for $q\neq1,2,\frac{3\pm\sqrt{5}}{2}$ by~(\ref{eq:c-by-continuants})
and Lemma~\ref{lem:gn-roots}. Hence $\mathbf{m}\in L\left(a/b\right)$.
We have $w(a/b,\mathbf{m})=1/b^{2}\neq1$, so $q=a/b$ is forbidden.

To find $a$ and $b$ satisfying~(\ref{eq:unit-norm}) and~(\ref{eq:sufficient-conditions})
we rewrite~(\ref{eq:unit-norm}) as
\begin{equation}
\left(a-\omega_{1}b\right)\left(a-\omega_{2}b\right)=\pm1\label{eq:unit-norm-rewritten}
\end{equation}
where $\omega_{1}=\frac{3+\sqrt{5}}{2}$ and $\omega_{2}=\frac{3-\sqrt{5}}{2}$.
Let $M>0$ be large. Since $\left\{ 1,\omega_{2}\right\} $ is a basis
of $\mathbf{Z}\left[\frac{1+\sqrt{5}}{2}\right]$, we can find $k$
large enough and integers $a,b$ such that
\[
a-b\omega_{2}=\left(\frac{1+\sqrt{5}}{2}\right)^{k}>M,
\]
in particular $b\neq0$. The norm of a unit is $\pm1$, so~(\ref{eq:unit-norm-rewritten})
holds. Therefore 
\[
\left|\frac{a}{b}-\omega_{1}\right|=\frac{1}{\left|b\right|}\left|a-b\omega_{1}\right|<\frac{1}{M}.
\]
In particular $\left|\frac{a}{b}-\omega_{1}\right|<\frac{1}{3}$ implies
$b\neq\pm1$ and $a/b\neq1,2$. Hence $q=a/b$ is forbidden and it
is arbitrarily close to $\omega_{1}$. The case of $\omega_{2}$ is
analogous.\hfill{}$\qed$

\medskip{}
We remark that for general $\mathbf{m}=(m_{0},m_{1},m_{2},m_{3},m_{4})$
with non-zero terms we would get
\[
c(a/b,\mathbf{m})=m_{4}+\frac{(m_{0}+m_{2})b^{2}+m_{0}m_{1}m_{2}ab}{b^{2}+(m_{0}m_{1}+m_{0}m_{3}+m_{2}m_{3})ab+m_{0}m_{1}m_{2}m_{3}a^{2}},
\]
so we would essentially have to solve~(\ref{eq:unit-norm-rewritten})
for
\[
\omega_{1,2}=\frac{1}{2}\left(-u-v_{1}-v_{2}\pm\sqrt{u^{2}+v_{1}^{2}+v_{2}^{2}-2v_{1}v_{2}+2uv_{1}+2uv_{2}}\right)
\]
where $v_{1}=\frac{1}{m_{0}m_{1}}$, $v_{2}=\frac{1}{m_{2}m_{3}}$
and $u=\frac{1}{m_{1}m_{2}}$. However, we cannot obtain an accumulation
point $\max(\omega_{1},\omega_{2})>\frac{3+\sqrt{5}}{2}$ in this
way, as by Proposition~\ref{prop:long-chain}, we would need to have
$(m_{i},m_{i+1},m_{i+2})=(\epsilon,-\epsilon,\epsilon)$ for $i=0$
or $1$ and $\epsilon=\pm1$. In the case $i=0$ we have $v_{1}=u=-1$
and $v_{2}=\frac{\epsilon}{m_{3}}$, so
\[
\omega_{1,2}=\frac{1}{2}\left(2-\frac{\epsilon}{m_{3}}+\sqrt{4+\frac{1}{m_{3}^{2}}}\right)\leq\frac{3+\sqrt{5}}{2}.
\]
The case $i=1$ is analogous.

Reaching higher accumulation points would therefore require the study
of longer loops. This, in turn, leads to searching for units in extensions
$\mathbf{Q}(\omega)$ of a higher degree, and these units would not,
in general, be of the form $a+b\omega$ that we require.

\section{Acknowledgments}

This research was supported by the grant 2021/41/B/ST1/00241 from the National Science
Centre, Poland. The author thanks Jerzy Kaczorowski for his remarks.

\end{document}